\documentclass[11pt, twoside]{article}
\usepackage{latexsym}
\usepackage{amsmath}
\usepackage{amssymb}
\usepackage[all]{xy}
\usepackage{amsfonts}
\usepackage{verbatim}
\usepackage{amsthm}
\usepackage{mathrsfs}
\usepackage{epsfig}
\usepackage{array}
\usepackage{stmaryrd}
\usepackage{graphicx,color}
\usepackage{xcolor}
\usepackage[colorlinks=true,linkcolor=blue,citecolor=blue]{hyperref}
\usepackage[nameinlink,noabbrev]{cleveref}
\usepackage{tikz}
\usetikzlibrary{arrows,calc}
\usepackage{tikz-cd}
\usepackage{mathdots}
\usepackage{float}
\usepackage{graphics}
\usepackage{pdflscape}

\usepackage{enumitem}
\usepackage[perpage,symbol]{footmisc}
\usepackage{setspace}
\numberwithin{equation}{section}

\DeclareMathOperator{\Dec}{Dec}

\newcommand{\Ftwo}{\mathbb F_2}
\newcommand{\cC}{\mathcal C}
\newcommand{\cP}{\mathcal P}
\newcommand{\cR}{\mathcal R}
\newcommand{\cT}{\mathcal T}
\newcommand{\DeltaHat}{\widehat{\Delta}}
\newcommand{\SigmaHat}{\widehat{\Sigma}}

\usepackage{bm}
\begin{document}
\baselineskip=15pt
\title{\Large{\bf A right pretriangulated category which is not right triangulated\footnotetext{Jing He is supported by the National Natural Science Foundation of China (Grant No. 12401045). Panyue Zhou is supported by the National Natural Science Foundation of China (Grant No. 12371034).
}}}
\medskip
\author{Jing He and Panyue Zhou}

\date{}

\maketitle
\def\blue{\color{blue}}
\def\red{\color{red}}

\setlist[enumerate]{leftmargin=2.2em,itemsep=0.25em,topsep=0.4em}
\setlist[itemize]{leftmargin=2.0em,itemsep=0.25em,topsep=0.4em}

\newtheorem{theorem}{Theorem}[section]
\newtheorem{proposition}[theorem]{Proposition}
\newtheorem{lemma}[theorem]{Lemma}
\newtheorem{corollary}[theorem]{Corollary}
\newtheorem{definition}[theorem]{Definition}
\newtheorem{remark}[theorem]{Remark}

\crefname{theorem}{Theorem}{Theorems}
\crefname{proposition}{Proposition}{Propositions}
\crefname{lemma}{Lemma}{Lemmas}
\crefname{corollary}{Corollary}{Corollaries}
\crefname{definition}{Definition}{Definitions}
\crefname{remark}{Remark}{Remarks}

\baselineskip=17pt
\parindent=0.5cm
\vspace{-6mm}

\begin{abstract}
\baselineskip=16pt
Chen, Liu, Lu, and Zhang recently constructed a pretriangulated category with invertible suspension in which Verdier's octahedral axiom fails. We introduce a general enlargement construction for right pretriangulated categories and show that it preserves axioms (RTR1)--(RTR3), while failure of (RTR4) is detected by the forgetful functor. Applied to their type $A_5$ example, the construction yields a right pretriangulated category that is not right triangulated. In this example, the suspension is faithful but not essentially surjective.
\\[0.2cm]
\end{abstract}

\pagestyle{myheadings}
\markboth{\rightline {\scriptsize   Jing He and Panyue Zhou}}
         {\leftline{\scriptsize A right pretriangulated category which is not right triangulated}}

\section{Introduction}

Right triangulated categories are one-sided analogues of triangulated categories in which the suspension is only required to be an additive endofunctor. They arise naturally in stable and quotient constructions where suspension is available in one direction but need not admit an inverse; see \cite{ABM98, BM94}. Their structure is governed by four axioms. The first three concern right triangles, rotation, and morphisms of right triangles, while the fourth is the octahedral axiom.

Throughout this paper, a triple satisfying (RTR1)--(RTR3) will be called a \emph{right pretriangulated category}, and one satisfying (RTR1)--(RTR4) a \emph{right triangulated category}. This terminology is used only in the one-sided setting and is unrelated to the notion of a pretriangulated dg category. The distinction is relevant here because the failure of the suspension to be invertible does not by itself force the failure of the right triangulated axioms; the essential issue is the octahedral axiom.

The independence of the octahedral axiom is known in the case of invertible suspension. Chen, Liu, Lu, and Zhang constructed a pretriangulated category satisfying Verdier's (TR1)--(TR3) but not (TR4) \cite[Theorem~1.2]{CLLZ26}. Their example is built from the category of finitely generated projective modules over the type $A_5$ preprojective algebra over $\mathbb F_2$, with suspension given by the autoequivalence induced by reflection of the Dynkin diagram.

The aim of this note is to turn their example into a genuinely one-sided one while retaining the failure of the octahedral axiom. More generally, let $(\cT,\Sigma,\Delta)$ be a right pretriangulated category and let
$$
F\colon\cT\longrightarrow\operatorname{fdVect}_k
$$
be an additive functor, where $\operatorname{fdVect}_k$ denotes the category of finite-dimensional $k$-vector spaces. We consider pairs $(X,U)$, where $X\in\cT$ and $U$ is a subspace of $F(X)$, with morphisms required to preserve the chosen subspaces. The suspension is defined by
$$
(X,U)\longmapsto(\Sigma X,0).
$$
If
$
X\xrightarrow{f}Y\xrightarrow{g}Z\xrightarrow{h}\Sigma X
$
is a right triangle and $F(f)U\subseteq V$, then the subspace attached to the third term is taken to be $F(g)V$. With this definition, the resulting category again satisfies (RTR1)--(RTR3).

The construction also preserves the failure of the octahedral axiom. Indeed, the zero-subspace embedding sends right triangles in the original category to right triangles in the enlarged category. Thus, if (RTR4) held in the enlarged category, forgetting the subspace data would give the corresponding octahedral completion in the original category. Equivalently, the decoration construction reflects the validity of (RTR4).

Applying this construction to the type $A_5$ example of Chen--Liu--Lu--Zhang gives the main result.

\begin{theorem}\label{thm:intro}
There exist an additive category $\cR$, an additive endofunctor $\SigmaHat\colon\cR\to\cR$, and a class of  right triangles satisfying \emph{(RTR1)}, \emph{(RTR2)}, and \emph{(RTR3)}, but not \emph{(RTR4)}. Moreover, $\SigmaHat$ is faithful and not essentially surjective. Hence $\cR$ is right pretriangulated but not right triangulated, and $\SigmaHat$ is not an autoequivalence.
\end{theorem}

The category of decorated objects is closely related to the classical subspace categories arising in representation theory; see, for example, \cite{DRSS99}. The novelty here lies not in the introduction of subspace data itself, but in showing that this construction is compatible with (RTR1)--(RTR3) and preserves the obstruction to the octahedral axiom when passing to a setting with noninvertible suspension.

This article is organised as follows. In Section 2, we recall the right triangulated axioms in the form used in this paper. In Section 3, we describe the type $A_5$ example of Chen--Liu--Lu--Zhang. In Section 4, we establish the general transfer construction. In the final section, we apply this construction to obtain the one-sided example with noninvertible suspension.

\section{Right pretriangulated and right triangulated categories}\label{sec:definitions}

In this section, we recall the basic terminology and axioms for right triangulated categories. We use the right-handed axioms as stated by Assem--Beligiannis--Marmaridis \cite[Definition~1.1]{ABM98}, which are dual to the left triangulated axioms of Beligiannis--Marmaridis \cite[Definition~2.2]{BM94}.

Let $\cC$ be an additive category and let $\Sigma\colon\cC\to\cC$ be an additive endofunctor. A \emph{sextuple} in $\cC$ is a sequence
$$
\xymatrix{
A\ar[r]^f&B\ar[r]^g&C\ar[r]^{h~}&\Sigma A
}
$$
with $A,B,C\in\cC$. A morphism of sextuples is a triple $(\alpha,\beta,\gamma)$ making the diagram
$$
\xymatrix{
A\ar[r]^f\ar[d]_{\alpha}&
B\ar[r]^g\ar[d]_{\beta}&
C\ar[r]^h\ar[d]_{\gamma}&
\Sigma A\ar[d]^{\Sigma\alpha}\\
A'\ar[r]^{f'}&
B'\ar[r]^{g'}&
C'\ar[r]^{h'}&
\Sigma A'
}
$$
commutative. It is an \emph{isomorphism of sextuples} if $\alpha$, $\beta$ and $\gamma$ are isomorphisms.

Let $\nabla$ be a class of sextuples in $\cC$. We say that $(\cC,\Sigma,\nabla)$ satisfies (RTR1)--(RTR3) if the following conditions hold.
\begin{enumerate}[label=\textup{(RTR\arabic*)},leftmargin=3.7em]

\item\label{item:RTR1}
The class $\nabla$ is closed under isomorphisms of sextuples. For every object $A\in\cC$,
$$
\xymatrix{
0\ar[r]&A\ar[r]^{1_A}&A\ar[r]&0
}
$$
belongs to $\nabla$, and every morphism $f\colon A\to B$ can be completed to a sextuple
$$
\xymatrix{
A\ar[r]^f&B\ar[r]^g&C\ar[r]^{h~}&\Sigma A
}
$$
in $\nabla$.

\item\label{item:RTR2}
If
$$
\xymatrix{
A\ar[r]^f&B\ar[r]^g&C\ar[r]^h&\Sigma A
}
$$
belongs to $\nabla$, then so does its forward rotation

$$
\xymatrix{
B\ar[r]^g&C\ar[r]^h&\Sigma A\ar[r]^{-\Sigma f}&\Sigma B.
}
$$

\item\label{item:RTR3}
Given two sextuples in $\nabla$ and morphisms $\alpha\colon A\to A'$ and $\beta\colon B\to B'$ satisfying $\beta f=f'\alpha$, there exists $\gamma\colon C\to C'$ such that

$$
\xymatrix{
A\ar[r]^f\ar[d]_{\alpha}&
B\ar[r]^g\ar[d]_{\beta}&
C\ar[r]^h\ar@{-->}[d]^{\gamma}&
\Sigma A\ar[d]^{\Sigma\alpha}\\
A'\ar[r]^{f'}&
B'\ar[r]^{g'}&
C'\ar[r]^{h'}&
\Sigma A'
}
$$

commutes.

\end{enumerate}

A triple $(\cC,\Sigma,\nabla)$ satisfying (RTR1)--(RTR3) is called a \emph{right pretriangulated category}, and the sextuples in $\nabla$ are called \emph{right triangles}. We shall write a right triangle simply as
$$
\xymatrix{
A\ar[r]^f&B\ar[r]^g&C\ar[r]^h&\Sigma A.
}
$$
A right pretriangulated category is \emph{right triangulated} if it also satisfies the following right octahedral axiom (see (RTR4)).
\begin{enumerate}[label=\textup{(RTR4)},leftmargin=3.5em]

\item\label{item:RTR4}{\bf (right octahedral axiom)} Suppose that
$$
\xymatrix{
A\ar[r]^f&B\ar[r]^g&C\ar[r]^h&\Sigma A,
}
~
\xymatrix{
A\ar[r]^{af}&X\ar[r]^d&Y\ar[r]^e&\Sigma A,
}
$$
and
$$
\xymatrix{
B\ar[r]^a&X\ar[r]^b&Z\ar[r]^c&\Sigma B
}
$$
are right triangles. Then there exist morphisms
$
s\colon C\to Y
$
and
$
t\colon Y\to Z
$
such that the diagram

$$
\xymatrix{
A\ar[r]^f\ar@{=}[d]&
B\ar[r]^g\ar[d]^a&
C\ar[r]^h\ar[d]^s&
\Sigma A\ar@{=}[d]\\
A\ar[r]^{af}&
X\ar[r]^d\ar[d]^b&
Y\ar[r]^e\ar[d]^t&
\Sigma A\ar[d]^{\Sigma f}\\
&
Z\ar@{=}[r]\ar[d]^c&
Z\ar[r]^c\ar[d]^{(\Sigma g)c}&
\Sigma B\\
&
\Sigma B\ar[r]^{\Sigma g}&
\Sigma C&
}
$$
commutes, and
$$
\xymatrix{
C\ar[r]^s&Y\ar[r]^t&Z\ar[r]^{(\Sigma g)c}&\Sigma C}
$$
is a right triangle.
\end{enumerate}
The endofunctor $\Sigma$ is called the \emph{suspension functor}.

\begin{remark}\label{rem1}
The main difference between a right triangulated category and a triangulated category lies in the suspension functor. In a right triangulated category $(\cC,\Sigma,\nabla)$, the functor $\Sigma$ is not required to be an autoequivalence. In fact, $(\cC,\Sigma,\nabla)$ is triangulated if and only if $\Sigma$ is an autoequivalence.
\end{remark}

We will use the following standard consequence of the first three axioms.

\begin{lemma}\label{lem:zero-compositions}
Let $(\cC,\Sigma,\Delta)$ be a right pretriangulated category. If
$$
\xymatrix{
A\ar[r]^f&B\ar[r]^g&C\ar[r]^h&\Sigma A
}
$$
is a right triangle, then
$gf=0$ and $hg=0$.
\end{lemma}

\begin{proof}
By (RTR2), it suffices to verify that $gf=0$. By (RTR1) and one application of (RTR2), $A\xrightarrow{1_A}A\longrightarrow0\longrightarrow\Sigma A$ is a right triangle. Applying (RTR3) to the commutative solid part of the following diagram, we obtain a morphism completing it to a commutative diagram:
$$\xymatrix{
A \ar[r]^1 \ar@{=}[d] & A \ar[r] \ar[d]^f & 0 \ar[r] \ar@{-->}[d] & \Sigma A \ar@{=}[d]\\
A \ar[r]^{f} & B \ar[r]^{g} & C \ar[r]^{h\hspace{1mm}} & \Sigma A}
$$
It follows that $gf=0$.
\end{proof}

\section{A seed from type $A_5$}\label{sec:seed}

Let $k=\Ftwo$. We first fix the quiver and the convention for composition of paths. Let $\overline{A_5}$ be the doubled quiver

$$
\begin{tikzcd}[column sep=3.4em]
1 \arrow[bend left=18]{r}{b_1}
& 2 \arrow[bend left=18]{l}{a_1} \arrow[bend left=18]{r}{b_2}
& 3 \arrow[bend left=18]{l}{a_2} \arrow[bend left=18]{r}{b_3}
& 4 \arrow[bend left=18]{l}{a_3} \arrow[bend left=18]{r}{b_4}
& 5 \arrow[bend left=18]{l}{a_4}.
\end{tikzcd}
$$
Thus $b_i\colon i\to i+1$ and $a_i\colon i+1\to i$ for $1\leq i\leq4$. Paths are composed from right to left. Following \cite[p.~158]{GelfandPonomarev1979} and \cite[Section~3.1]{CLLZ26}, set
$$
   \Lambda=\Pi_k(A_5)=k\overline{A_5}/I,
$$
where $I$ is generated by
$
   a_1\circ b_1,~~
   b_i\circ a_i+a_{i+1}\circ b_{i+1}~(1\leq i\leq3),~~
   b_4\circ a_4.
$
Thus the length-two backtracking paths at the boundary vertices vanish, while at each internal vertex the two backtracking paths are identified. Since $k=\Ftwo$, the usual difference appearing in the preprojective relation agrees with the sum displayed above.

Let $e_1,\ldots,e_5$ denote the images of the stationary paths. These are pairwise orthogonal primitive idempotents and
$$
   1_\Lambda=e_1+\cdots+e_5.
$$
The algebra $\Lambda$ is $35$-dimensional and self-injective \cite[Proposition~3.1]{CLLZ26}. We write $\operatorname{mod}!\text{-}\Lambda$ for the category of finite-dimensional right $\Lambda$-modules and $\cP(\Lambda)$ for its full subcategory of finitely generated projective modules. For $1\leq i\leq5$, put
$
   P_i=e_i\Lambda.
$
Then $P_1,\ldots,P_5$ are the indecomposable projective right $\Lambda$-modules, and
$$
   \Lambda_\Lambda\cong\bigoplus_{i=1}^5P_i.
$$
Reflection of the Dynkin diagram induces an involutive algebra automorphism
$$
   \nu\colon\Lambda\longrightarrow\Lambda
$$
determined by
$$
   \nu(e_i)=e_{6-i},
  ~~
   \nu(a_i)=b_{5-i},
   ~~
   \nu(b_i)=a_{5-i}.
$$
For a right $\Lambda$-module $P$, let $P_\nu$ denote the vector space $P$ equipped with the twisted right action
$$
   x\mathbin{\cdot_\nu}\lambda=x\nu(\lambda).
$$
Twisting by $\nu$ yields an autoequivalence
$$
   \Sigma=(-)_\nu\colon\cP(\Lambda)\longrightarrow\cP(\Lambda).
$$
We now recall the result from \cite{CLLZ26} that provides the starting point for our construction.

\begin{theorem}{\rm \cite[Theorem~1.2]{CLLZ26}}\label{thm:CLLZ}
There is a class $\Delta_\varepsilon$ of distinguished candidate triangles on $(\cP(\Lambda),\Sigma)$ satisfying Verdier's axioms \emph{(TR1)--(TR3)} but not \emph{(TR4)}.
\end{theorem}

We next translate this statement into the language of right triangulated categories.

\begin{proposition}\label{prop:seed}
The triple
$$
   \bigl(\cP(\Lambda),\Sigma,\Delta_\varepsilon\bigr)
$$
is right pretriangulated and does not satisfy \emph{(RTR4)}.
\end{proposition}

\begin{proof}
The closure under isomorphisms and the completion of every morphism to a candidate triangle follow directly from Verdier's (TR1). It remains only to verify the trivial sequence required in (RTR1).

Let $A\in\cP(\Lambda)$. Since $\Sigma$ is an autoequivalence, there is an object $X$ with $\Sigma X\cong A$. By Verdier's (TR1),
$$
   X\xrightarrow{1_X}X\longrightarrow0\longrightarrow\Sigma X
$$
is a distinguished candidate triangle. Applying the rotation axiom twice gives
$$
   0\longrightarrow\Sigma X\xrightarrow{-1_{\Sigma X}}\Sigma X\longrightarrow0.
$$
This triangle is isomorphic to
$$
   0\longrightarrow A\xrightarrow{1_A}A\longrightarrow0.
$$
Hence (RTR1) holds. Verdier's rotation and morphism axioms give (RTR2) and (RTR3), respectively.

Assume, to the contrary, that (RTR4) also holds. Since $\Sigma$ is an autoequivalence, the resulting right triangulated category is a triangulated category (see Remark \ref{rem1}). In particular, $\Delta_\varepsilon$ would satisfy Verdier's octahedral axiom, contrary to \cref{thm:CLLZ}. Thus (RTR4) fails.
\end{proof}

We shall not recall the construction of $\Delta_\varepsilon$ or the fixed-boundary obstruction from \cite{CLLZ26}. In what follows, \cref{prop:seed} will be used only as a non-octahedral right pretriangulated seed.

\section{The decoration construction}\label{sec:decoration}

Throughout this section, $k$ denotes a field and $(\cT,\Sigma,\Delta)$ is a right pretriangulated category.
Let
$$
F\colon\cT\longrightarrow\operatorname{fdVect}_k
$$
be an additive functor, where $\operatorname{fdVect}_k$ denotes the category of finite-dimensional $k$-vector spaces.

\begin{definition}\label{def:decorated-category}
The category $\Dec_F(\cT)$ is defined as follows.
\begin{enumerate}[label=\textup{(\roman*)},leftmargin=2.8em]
\item The objects are pairs $(X,U)$, where $X\in\cT$ and $U$ is a vector subspace of $F(X)$.
\item A morphism
$$
f\colon(X,U)\longrightarrow(Y,V)
$$
is a morphism $f\colon X\to Y$ in $\cT$ such that $F(f)(U)\subseteq V$.
\end{enumerate}
\end{definition}
The definition is compatible with composition: if $f$ and $g$ preserve the relevant subspaces, then so does $gf$.  Moreover, $\Dec_F(\cT)$ is additive.  Its morphisms form additive subgroups of the corresponding Hom groups in $\cT$, its zero object is $(0,0)$, and, under the canonical identification $F(X\oplus Y)\cong F(X)\oplus F(Y)$, its biproducts are
$$
   (X,U)\oplus(Y,V)=(X\oplus Y,U\oplus V).
$$
Two elementary functors will be used repeatedly.  The first is the forgetful functor
\[
   Q\colon\Dec_F(\cT)\longrightarrow\cT,
  ~~Q(X,U)=X,
\]
and the second is the zero-decoration embedding
\[
   J\colon\cT\longrightarrow\Dec_F(\cT),
   ~~ J(X)=(X,0).
\]
The functor $J$ is fully faithful, and $QJ=1_{\cT}$.

Define an additive endofunctor
\[
   \SigmaHat(X,U)=(\Sigma X,0),
  ~~
   \SigmaHat(f)=\Sigma f.
\]
Then
\begin{equation}\label{eq:compatibility}
   Q\SigmaHat=\Sigma Q,
   ~~
   \SigmaHat J=J\Sigma.
\end{equation}

We now define the  right triangles.  Let
\begin{equation}\label{eq:underlying-triangle}
   X\xrightarrow{f}Y\xrightarrow{g}Z\xrightarrow{h}\Sigma X
\end{equation}
be right triangle in $\cT$, and choose subspaces
\[
   U\subseteq F(X),
  ~~
   V\subseteq F(Y)
\]
with $F(f)U\subseteq V$.  We associate to this data the sequence
\begin{equation}\label{eq:decorated-triangle}
   (X,U)\xrightarrow{f}(Y,V)
   \xrightarrow{g}(Z,F(g)V)
   \xrightarrow{h}(\Sigma X,0).
\end{equation}
The last arrow is a morphism in $\Dec_F(\cT)$ because \cref{lem:zero-compositions} gives $hg=0$, and hence
\[
   F(h)F(g)V=0.
\]

\begin{definition}\label{def:DeltaHat}
Let $\DeltaHat$ be the isomorphism closure of all sequences of the form \eqref{eq:decorated-triangle}.  Its members are called right triangles.
\end{definition}

We can now state the transfer result.  Its second assertion is the mechanism that carries the octahedral obstruction from the original category to the decorated one.

\begin{proposition}\label{prop:transfer}
If $(\cT,\Sigma,\Delta)$ is right pretriangulated, then
\[
   \bigl(\Dec_F(\cT),\SigmaHat,\DeltaHat\bigr)
\]
is right pretriangulated.  Moreover,
\begin{equation}\label{eq:RTR4-reflection}
   \bigl(\Dec_F(\cT),\SigmaHat,\DeltaHat\bigr)
   \text{ satisfies \emph{(RTR4)}}
  ~~\Longrightarrow~~
   (\cT,\Sigma,\Delta)
   \text{ satisfies \emph{(RTR4)}}.
\end{equation}
\end{proposition}

\begin{proof}
We first verify the axioms for the representatives in \eqref{eq:decorated-triangle}.  Because the axioms are invariant under isomorphism, the same conclusions then hold for their isomorphism closure.  In the verification of (RTR3), a square involving triangles in the isomorphism closure is transported to a square between representatives, completed there, and transported back.

For (RTR1), the trivial right triangle of $(X,U)$ is obtained from
\[
   0\longrightarrow X\xrightarrow{1_X}X\longrightarrow0
\]
by using the decorations $0$ and $U$ on the first two terms.  Since $F(1_X)U=U$, this gives
\[
   0\longrightarrow(X,U)\xrightarrow{1}(X,U)\longrightarrow0.
\]
Now let $f\colon(X,U)\to(Y,V)$ be any morphism in $\Dec_F(\cT)$.  Complete the underlying map $f\colon X\to Y$ to a right triangle \eqref{eq:underlying-triangle}.  The condition $F(f)U\subseteq V$ then makes \eqref{eq:decorated-triangle} a right triangle completing $f$.

For (RTR2), the forward rotation of \eqref{eq:decorated-triangle} has underlying right triangle
\[
   Y\xrightarrow{g}Z\xrightarrow{h}\Sigma X
   \xrightarrow{-\Sigma f}\Sigma Y.
\]
The first two decorations are $V$ and $F(g)V$, and the third decoration prescribed by \cref{def:DeltaHat} is
\[
   F(h)F(g)V=0.
\]
Hence the rotated decorated triangle is
\[
   (Y,V)\xrightarrow{g}(Z,F(g)V)
   \xrightarrow{h}(\Sigma X,0)
   \xrightarrow{-\Sigma f}(\Sigma Y,0),
\]
and is a right triangle.

For (RTR3), consider a commutative square between the first morphisms of two decorated right triangles of the form \eqref{eq:decorated-triangle}.  On the underlying right triangles, \emph{(RTR3)} in $\cT$ gives a filler $\gamma\colon Z\to Z'$ satisfying
\[
   \gamma g=g'\beta.
\]
Since $\beta\colon(Y,V)\to(Y',V')$ is a decorated morphism,
\[
   F(\beta)V\subseteq V'.
\]
Therefore
\[
   F(\gamma)(F(g)V)
      =F(g')F(\beta)V
      \subseteq F(g')V',
\]
so $\gamma$ is also a morphism in $\Dec_F(\cT)$.  This proves (RTR3).

Finally, assume that the decorated category satisfies (RTR4).  Start with an arbitrary (RTR4) datum in $\cT$ and apply $J$ to every object and morphism.  Because all decorations are zero, each of the three given right triangles is sent to a right triangle in $\Dec_F(\cT)$, and \eqref{eq:compatibility} identifies the suspended morphisms correctly.  An (RTR4) completion in $\Dec_F(\cT)$ therefore forgets under $Q$ to morphisms $s$ and $t$ satisfying the required commutative relations in $\cT$.  The additional row is a right triangle after applying $Q$: this is immediate for the representatives in \eqref{eq:decorated-triangle} and follows for their isomorphism closure because $\Delta$ is closed under isomorphisms.  Thus (RTR4) holds in $\cT$.
\end{proof}

\begin{remark}\label{rem:reflection-only}
Only reflection of \emph{(RTR4)} is asserted.  The proof does not claim that an octahedral completion in $\cT$ lifts to every decorated octahedral datum.
\end{remark}

\section{The example}\label{sec:example}

We now apply the transfer result to the seed in \cref{prop:seed}.  Let
\[
   F\colon\cP(\Lambda)\longrightarrow\operatorname{fdVect}_{\Ftwo}
\]
be the underlying-vector-space functor and set
\[
   \cR=\Dec_F\bigl(\cP(\Lambda)\bigr).
\]
Thus an object of $\cR$ is a finitely generated projective right $\Lambda$-module $P$ together with a chosen $\Ftwo$-subspace $U$ of its underlying vector space.  In this case, the decorated suspension takes the concrete form
\[
   \SigmaHat(P,U)=(P_\nu,0).
\]

\begin{theorem}\label{thm:main}
With the right triangles obtained from $\Delta_\varepsilon$ by \cref{def:DeltaHat}, the triple
\[
   (\cR,\SigmaHat,\DeltaHat)
\]
is a right pretriangulated category which is not right triangulated.  The endofunctor $\SigmaHat$ is faithful but not essentially surjective, and hence is not an autoequivalence.
\end{theorem}

\begin{proof}
By \cref{prop:transfer}, the triple $(\cR,\SigmaHat,\DeltaHat)$ satisfies (RTR1)--(RTR3).  If it satisfied \emph{(RTR4)}, then \eqref{eq:RTR4-reflection} would imply (RTR4) for $(\cP(\Lambda),\Sigma,\Delta_\varepsilon)$, contrary to \cref{prop:seed}.  Thus $\cR$ is right pretriangulated but not right triangulated.

Since $\Sigma=(-)_\nu$ is an autoequivalence, it is faithful.  If $f$ is a morphism in $\cR$ and $\SigmaHat(f)=0$, then
\[
   \Sigma(Qf)=Q(\SigmaHat f)=0.
\]
Faithfulness of $\Sigma$ gives $Qf=0$, hence $f=0$.  Thus $\SigmaHat$ is faithful.

To see that $\SigmaHat$ is not essentially surjective, let
\[
   P_1=e_1\Lambda
\]
and choose the nonzero subspace
\[
   U=\Ftwo e_1\subseteq F(P_1).
\]
Every object in the image of $\SigmaHat$ has zero decoration.  Suppose that $(P_1,U)$ were isomorphic to a zero-decorated object $(Y,0)$.  Such an isomorphism
\[
   a\colon(P_1,U)\xrightarrow{\sim}(Y,0)
\]
must satisfy $F(a)(U)=0$.  But the underlying map $a\colon P_1\to Y$ is an isomorphism, so $F(a)$ is an isomorphism of vector spaces.  This forces $U=0$, a contradiction.  Hence $(P_1,U)$ is not in the essential image of $\SigmaHat$.
\end{proof}

\begin{corollary}\label{cor:existence}
There exists a right pretriangulated category with nonzero, noninvertible suspension which is not right triangulated.
\end{corollary}

\begin{proof}
Apply \cref{thm:main}.  The suspension is nonzero because its underlying functor is the autoequivalence $(-)_\nu$ on the nonzero category $\cP(\Lambda)$.
\end{proof}

\begin{remark}\label{rem:scope}
The failure of \emph{(RTR4)} in \cref{thm:main} is inherited from the type $A_5$ construction of Chen--Liu--Lu--Zhang.  The decoration construction is not a second independent source of a non-octahedral pretriangulation.  Its role is to convert the known example with invertible suspension into a right pretriangulated example with noninvertible suspension while retaining the same obstruction to the right octahedral axiom.
\end{remark}

\vspace{3mm}

\hspace{-5.5mm}\textbf{Data Availability}\hspace{2mm} Data sharing not applicable to this article as no datasets were generated or analysed during
the current study.
\vspace{3mm}

\hspace{-5.5mm}\textbf{Conflict of Interests}\hspace{2mm} The authors declare that they have no conflicts of interest to this work.

\vspace{3mm}

%\hspace{-5.5mm}{\bf Declaration of generative AI and AI-assisted technologies in the manuscript preparation process:}
%During the preparation of this work, the authors used ChatGPT-5.6Sol for language polishing and proofreading. The authors reviewed and edited the output as needed and take full responsibility for the content of the published article.

\hspace{-5mm}\textbf{Jing He}\\
School of Mathematics and Statistics, Hunan University of Technology and Business, 410205 Changsha, Hunan P. R. China\\
E-mail: jinghe1003@163.com\\[0.4cm]
\textbf{Panyue Zhou}\\
School of Mathematics and Statistics, Changsha University of Science and Technology, 410114 Changsha, Hunan, P. R. China\\
E-mail: panyuezhou@163.com

\end{document}